\documentclass[12pt]{article}
\usepackage{graphicx}

\usepackage{epstopdf}
\usepackage{float}
\usepackage{float}
\usepackage{color}
\usepackage{amssymb}
\usepackage{algorithm}
\usepackage{algorithmic}

\usepackage{mathtools}
\usepackage[thmmarks,amsmath]{ntheorem}
\usepackage{bm}

\usepackage{extarrows}
\usepackage[hang,flushmargin]{footmisc}

\usepackage{mathrsfs}
\usepackage[font=footnotesize,skip=0pt,textfont=rm,labelfont=rm]{caption,subcaption}

\usepackage{booktabs} 
\usepackage{tocloft}
\usepackage{hyperref}
\hypersetup{hypertex=true,
	colorlinks=true,
	linkcolor=blue,
	anchorcolor=blue,
	citecolor=blue}
\usepackage{caption}
\usepackage{cite}

{
	\theoremstyle{nonumberplain}
	\theoremheaderfont{\bfseries}
	\theorembodyfont{\normalfont}
	\theoremsymbol{\mbox{$\Box$}}
	\newtheorem{proof}{Proof}
}


\usepackage{theorem}

\newtheorem{theorem}{Theorem}[section]

\newtheorem{lemma}{Lemma}[section]
\newtheorem{definition}{Definition}[section]

\newtheorem{corollary}{Corollary}[section]

\newcommand{\Rmnum}[1]{\uppercase\expandafter{\romannumeral #1}}

{
	\theoremheaderfont{\bfseries}
	\theorembodyfont{\rm}
	
}
\graphicspath{{figure/}}                                 
\usepackage[a4paper]{geometry}                           
\begin{document}
	\title{\textbf{Characterizing chordal $\llcorner$-EPG graphs via admissible clique tree orientations}}
	
	\author{Bin Sun$^{1}$, Shou-Jun Xu$^{1}$, Yu Yang$^{2,\dagger}$\\
	{\rmfamily\small $^{1}$School of Mathematics and Statistics, Lanzhou University, Lanzhou, 730000, PR China }\\
	{\rmfamily\small $^{2}$School of Mathematics and Statistics, Huaibei Normal University, Huaibei, 235000, PR China}\\
}

\renewcommand{\thefootnote}{\fnsymbol{footnote}}

\footnotetext[2]{Corresponding author. E-mail address: yyang@chnu.edu.cn (Y. Yang).}
	
	\date{}

	\maketitle
    
	\textbf{Abstract:} A $k$-bend path is a non-self-intersecting polyline that lies on a grid and consists of at most $k+1$ axis-parallel line segments. An $\llcorner$-EPG graph is a graph whose vertices can be represented by 1-bend paths on a grid, where each path is either $\llcorner$, $\shortmid$, or $\text{-}$, such that two vertices are adjacent if and only if the corresponding paths share at least one grid edge. Characterizing chordal $\llcorner$-EPG graphs was explicitly posed as an open problem by Cameron, Chaplick, and Hoàng. We resolve this by giving two equivalent characterizations of connected chordal $\llcorner$-EPG graphs. The first assigns to every connected chordal $\llcorner$-EPG graph an ordered clique partition tree of its vertex set, where the underlying tree is the bipartite incidence graph of the horizontal and vertical edge-intersection components. The second is stated purely in terms of maximal cliques: a connected chordal graph is $\llcorner$-EPG if and only if some clique tree admits an admissible partial orientation. A reduction lemma replaces local clique labels by maximal cliques, linking the two characterizations. For a prescribed clique tree, the existence of an admissible partial orientation reduces to 2-SAT and is decided in $O(nM^2)$ time, where $n$ and $M$ are the numbers of vertices and maximal cliques. Finally, we show that strong chordality is not sufficient: we exhibit a strongly chordal graph that is a minimal forbidden induced subgraph for $\llcorner$-EPG, and a family of split graphs with a unique clique tree for which membership in $\llcorner$-EPG reduces to a neighborhood condition.
	\vspace{1cm}	
	

\section{Introduction}
\par Intersection graphs of various types of objects have been extensively studied in the last sixty years. A $k$-bend path is a non-self-intersecting polyline that lies on a grid and consists of at most $k+1$ axis-parallel line segments. For 1-bend paths, there are six possible shapes: $\llcorner$, $\lrcorner$, $\ulcorner$, $\urcorner$, $\shortmid$, and $\text{-}$. Golumbic et al. \cite{golumbic2009edge} introduced the {\em Edge intersection graph of Paths on a Grid} ({\em EPG graph}) whose vertices can be represented by paths on a grid such that two vertices are adjacent if and only if the corresponding paths share at least one grid edge. The authors also showed that every graph is an EPG graph. If a graph admits such a representation in which each path has at most $k$ bends, then it is called a {\em $B_{k}$-EPG graph}. Specifically, an {\em $\llcorner$-EPG} graph is a $B_1$-EPG graph whose vertices can be represented by 1-bend paths on a grid, where each path is either $\llcorner$, $\shortmid$, or $\text{-}$.

\par EPG graphs were introduced to study electrical networks \cite{sinden1966topology} and have practical applications in areas such as circuit layout, where wires can be modeled as paths on a grid.

\par In this paper, all graphs are simple, i.e., they contain no loops or multiple edges. Let $G=(V(G),E(G))$ be a graph. We denote by $N(v)$ the {\em open neighborhood} of a vertex $v\in V(G)$, and $N[v]=N(v)\cup\{v\}$ is the {\em closed neighborhood} of $v$. The {\em degree} of $v$ is denoted by $d(v)=|N(v)|$. Let $\Delta(G)$ and $\delta(G)$ denote the {\em maximum degree} and {\em minimum degree} of $G$, respectively. We say that $G'$ is an {\em induced subgraph} of $G$ if $V(G')\subseteq V(G)$ and $E(G')=\{uv\mid uv\in E(G)\text{ and }u,v\in V(G')\}$. Two vertices $u$ and $v$ are {\em incompatible} if $N(u)\nsubseteq N(v)$ and $N(v)\nsubseteq N(u)$. A vertex $v$ is {\em simplicial} if $N(v)$ is a clique. A graph $G$ is {\em chordal} if and only if every induced subgraph of $G$ has a simplicial vertex. A graph is {\em split} if its vertex set has a partition into a clique and an independent set. 

\par A {\em clique tree} of a connected chordal graph $G$ is a tree $\mathcal{T}$ with node set $\mathcal{K}(G)$ such that, for every $v\in V(G)$, the nodes containing $v$ induce a connected subtree, denoted by $\mathcal{T}_v$. This is the {\em running-intersection property}. The classical subtree representation of chordal graphs provides such trees \cite{gavril1974intersection}. For a clique tree edge $KQ$, we call $S_{KQ}=K\cap Q$ its {\em separator}. Every such separator is nonempty when $G$ is connected: an empty separator would split the clique tree into two sides with disjoint vertex sets and no graph edge between them. For an interval graph, its maximal cliques admit a linear order in which the cliques containing any fixed vertex occur consecutively \cite{fulkerson1965incidence}.

\par Having established the necessary terminology, we now state the main problem. Asinowski and Ries \cite{asinowski2012some} characterized $B_1$-EPG graphs within chordal bull-free graphs, chordal claw-free graphs, chordal diamond-free graphs, and special cases of split graphs. Cameron et al. \cite{cameron2016edge} gave characterizations and polynomial-time recognition algorithms for gem-free split $\llcorner$-EPG graphs and bull-free split $\llcorner$-EPG graphs, and showed that recognizing $\llcorner$-EPG graphs is NP-complete in general. Deniz et al. \cite{deniz2018split,deniz2020some} further studied subclasses of split $B_1$-EPG graphs. All of these results rely on subgraph constraints. Further algorithmic and hardness results on subclasses of $B_1$-EPG graphs appear in \cite{Bessy2020independent,silva2025clique,wang2025algorithm,epstein2020hardness,epstein2013approximation,wang2026domination}. The present paper removes these restrictions and characterizes the full class of connected chordal $\llcorner$-EPG graphs.

\par The paper is organized as follows. Section 2 develops the ordered clique partition trees and their geometric realization. Section 3 proves the characterization by clique tree. Section 4 gives an example of this process: starting from its $\llcorner$-EPG representation, construct $T_{\mathcal{R}}$; then construct $\mathcal{U}$; and finally construct $\mathcal{T}$. Section 5 gives the complete 2-SAT reduction for deciding whether a given clique tree admits an admissible partial orientation. Section 6 presents a strongly chordal graph $H$ that is not an $\llcorner$-EPG graph, together with a class of split graphs with a unique clique tree admitting an admissible partial orientation.

\section{Ordered clique partition tree}
\par We present the first characterization of connected chordal $\llcorner$-EPG graphs, obtained by constructing an ordered clique partition tree of $G$. We next show that allowing straight paths does not enlarge the class.
\begin{lemma}\label{le2.1}
Every $\llcorner$-EPG graph has a representation in which both arms of every path have positive length and all endpoints have integer coordinates.
\end{lemma}

\begin{proof}
\par Start with a grid representation whose endpoints have integer coordinates. For a horizontal straight path, shift its left endpoint to the right by an amount between 0 and $1/4$, choosing a new column that no existing vertical arm uses. Since the grid has infinitely many columns and the representation uses only finitely many, such a column always exists. Choose different new columns for different horizontal straight paths. Add a short upward arm at the new left endpoint. For a vertical straight path, move its bottom endpoint upwards by an amount between 0 and $1/4$; choose a new row distinct from every existing horizontal row and from the other new rows, and add a short rightward arm there.

\par The original nontrivial segments have length at least one. Every original positive overlap also has length at least one. Moving left endpoints to the right or bottom endpoints upwards by an amount between 0 and $1/4$ therefore preserves every positive overlap. Shortening cannot create a new overlap. Each added vertical arm lies on a private column, and each added horizontal arm lies on a private row; hence no added arm shares a positive-length segment with another path. Choose their lengths to be positive rational numbers. Intersections between perpendicular arms consist only of points and do not create graph edges. Finally, multiply all coordinates by a common positive integer that clears their denominators.
\end{proof}

\begin{figure}
\centering
  \setlength{\unitlength}{1bp}%
  \begin{picture}(439.71, 127.13)(0,0)
  \put(0,0){\includegraphics{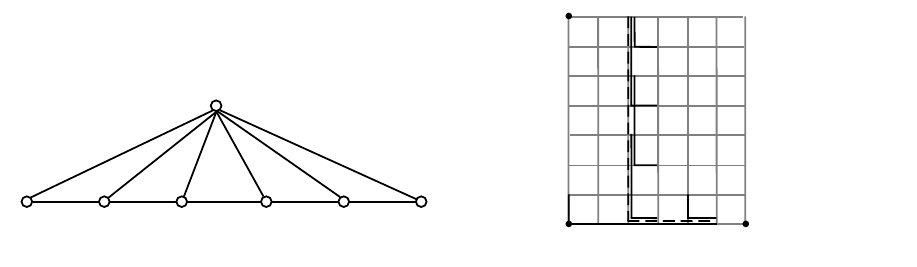}}
  \put(109.43,76.50){\fontsize{9.72}{11.66}\selectfont $v$}
  \put(5.67,20.17){\fontsize{9.72}{11.66}\selectfont $u_0$}
  \put(45.20,20.17){\fontsize{9.72}{11.66}\selectfont $u_1$}
  \put(83.08,20.17){\fontsize{9.72}{11.66}\selectfont $u_2$}
  \put(122.28,20.17){\fontsize{9.72}{11.66}\selectfont $u_3$}
  \put(158.52,20.50){\fontsize{9.72}{11.66}\selectfont $u_4$}
  \put(197.72,19.84){\fontsize{9.72}{11.66}\selectfont $u_5$}
  \put(101.04,7.76){\fontsize{9.72}{11.66}\selectfont $F_0$}
  \put(291.29,7.76){\fontsize{9.72}{11.66}\selectfont $\mathcal{R}=\langle F_0,P\rangle$}
  \put(331.84,26.57){\fontsize{9.72}{11.66}\selectfont $P_{u_0}$}
  \put(304.75,36.77){\fontsize{9.72}{11.66}\selectfont $P_{u_2}$}
  \put(314.80,47.98){\fontsize{9.72}{11.66}\selectfont $P_{u_3}$}
  \put(314.84,77.14){\fontsize{9.72}{11.66}\selectfont $P_{u_4}$}
  \put(316.21,106.05){\fontsize{9.72}{11.66}\selectfont $P_{u_5}$}
  \put(256.05,34.34){\fontsize{9.72}{11.66}\selectfont $P_{u_1}$}
  \put(256.05,12.29){\fontsize{9.72}{11.66}\selectfont (2,2)}
  \put(249.64,111.25){\fontsize{9.72}{11.66}\selectfont (2,9)}
  \put(359.20,12.29){\fontsize{9.72}{11.66}\selectfont (8,2)}
  \end{picture}%
  \caption{The graph $F_0$ and its $\llcorner$-EPG representation $\mathcal{R}$; the dashed $\llcorner$-shape in $\mathcal{R}$ is $P_v$.}
  \label{fig1}
\end{figure}

\par Accordingly, we denote by $P_v$ the path representing a vertex $v$, and let $P_v$ be
\[
 [(x_v,y_v),(r_v,y_v)]\cup[(x_v,y_v),(x_v,t_v)],\qquad r_v>x_v,\quad t_v>y_v.
\]
Its bend is $(x_v,y_v)$. Let $\mathcal{R}=\langle G,P\rangle$ be the $\llcorner$-EPG representation of $G$. We call a representation {\em non-degenerate} if every path has two arms of positive length. Moreover, in what follows, every {\em component} is understood to be a connected component, and {\em a horizontal (resp. vertical) directional component} is a connected component of $G_{\mathcal{H}}$ (resp. $G_{\mathcal{V}}$). All $\llcorner$-EPG representations below are assumed to be non-degenerate by default.

\begin{lemma}\label{le2.2}
In an $\llcorner$-EPG representation, every nonempty clique has a common grid edge. Let $G_{\mathcal{H}}$ (resp. $G_{\mathcal{V}}$) be the graphs on $V(G)$ in which two vertices are adjacent if their corresponding paths share a horizontal (resp. vertical) grid edge. If $H$ is a component of $G_{\mathcal{H}}$, then $G[H]=G_{\mathcal{H}}[H]$ and $G[H]$ is an interval graph. The analogous statements hold for every component of $G_{\mathcal{V}}$.
\end{lemma}
\begin{proof}
\par Let $C$ be a nonempty clique. Suppose first that all its horizontal arms lie on one row. Any two paths in $C$ either share a horizontal edge or share a vertical edge. In the latter case their columns are equal, and their rows already are equal, so they have the same bend and also share a horizontal edge. Thus all their horizontal intervals pairwise overlap in positive length. If $p$ is the largest left endpoint and $q$ the smallest right endpoint, pairwise positive overlap gives $p<q$. Since the endpoints are integers, the intervals have a common grid edge.

\par Otherwise, choose $u,v\in C$ with $y_u\ne y_v$. Their adjacency must be vertical, so $x_u=x_v$. If some $w\in C$ had $x_w\ne x_u$, its adjacencies to both $u$ and $v$ would have to be horizontal, forcing $y_w=y_u=y_v$, a contradiction. All vertical arms of $C$ therefore lie on one column. For two of these paths on different rows, adjacency is vertical; for two on the same row, their bends are equal and their positive vertical arms overlap. The preceding endpoint argument now gives a common vertical grid edge.

\par Next let $H$ be a component of $G_{\mathcal{H}}$. Horizontal adjacency preserves the row, so all paths for vertices of $H$ lie on one horizontal row. If $P_u,P_v$ with $u,v\in H$ share a vertical edge, their columns are equal. Since their rows are also equal, they have the same bend and share a horizontal edge as well. This proves $G[H]=G_{\mathcal{H}}[H]$. Replacing each horizontal segment $[x_v,r_v]$ by the closed interval $[x_v,r_v-1/4]$ converts positive-length overlap into ordinary interval intersection. These intervals represent $G[H]$. Interchanging horizontal and vertical coordinates proves the remaining assertion.
\end{proof}

\par Let $\mathcal{H}$ and $\mathcal{V}$ be the component families of $G_{\mathcal{H}}$ and $G_{\mathcal{V}}$, including singleton components. Define the bipartite incidence graph $T_{\mathcal{R}}$ on $\mathcal{H}\mathbin{\dot\cup}\mathcal{V}$ by joining $H\in\mathcal{H}$ to $Q\in\mathcal{V}$ exactly when $H\cap Q\ne\emptyset$. For such an edge, let $C_{HQ}=H\cap Q$. Every vertex of $G$ belongs to exactly one such set. All paths in $C_{HQ}$ have the same bend, so $C_{HQ}$ is a clique. 

\par Fig. \ref{fig1} shows the graph $F_0$, obtained by joining a vertex to all vertices of the path $P_6$, together with its $\llcorner$-EPG representation. This example will serve as a running example throughout the paper. Fig. \ref{fig2} shows the incidence graph $T_{\mathcal{R}}$ of the directional components obtained from the $\llcorner$-EPG representation of $F_0$.

\begin{figure}
\centering
  \setlength{\unitlength}{1bp}%
 \begin{picture}(341.55, 108.41)(0,0)
  \put(0,0){\includegraphics{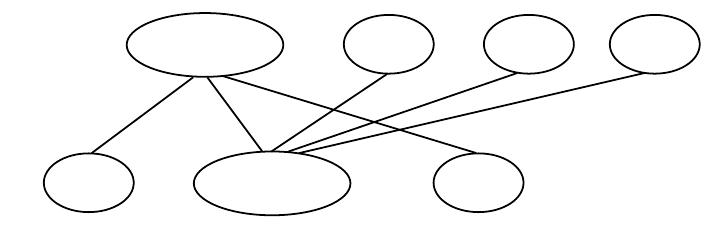}}
  \put(68.92,84.76){\fontsize{9.72}{11.66}\selectfont $u_0,u_1,u_2,v$}
  \put(100.23,17.70){\fontsize{9.72}{11.66}\selectfont $u_2,u_3,u_4,u_5,v$}
  \put(180.69,84.76){\fontsize{9.72}{11.66}\selectfont $u_3$}
  \put(247.95,84.76){\fontsize{9.72}{11.66}\selectfont $u_4$}
  \put(308.37,84.76){\fontsize{9.72}{11.66}\selectfont $u_5$}
  \put(223.78,18.31){\fontsize{9.72}{11.66}\selectfont $u_0$}
  \put(36.69,18.31){\fontsize{9.72}{11.66}\selectfont $u_1$}
  \put(44.77,85.81){\fontsize{9.72}{11.66}\selectfont $H_1$}
  \put(149.20,84.84){\fontsize{9.72}{11.66}\selectfont $H_2$}
  \put(217.31,84.20){\fontsize{9.72}{11.66}\selectfont $H_3$}
  \put(277.91,83.35){\fontsize{9.72}{11.66}\selectfont $H_4$}
  \put(79.67,17.52){\fontsize{9.72}{11.66}\selectfont $Q_1$}
  \put(193.88,17.91){\fontsize{9.72}{11.66}\selectfont $Q_2$}
  \put(5.67,18.88){\fontsize{9.72}{11.66}\selectfont $Q_3$}
  \end{picture}%
  \caption{The bipartite incidence graph $T_{\mathcal{R}}$ of the horizontal and vertical directional components, constructed from the representation $\mathcal{R}=\langle F_0,P\rangle$. The horizontal segments of $P_{u_0}$, $P_{u_1}$, $P_{u_2}$, and $P_v$ on row $y=2$ form a horizontal directional component $H_1$. On rows $y=4,6,8$, the horizontal segments of $P_{u_3}$, $P_{u_4}$, and $P_{u_5}$ form horizontal directional components $H_2$, $H_3$, and $H_4$, respectively. On column $x=4$, the vertical segments of $P_v$, $P_{u_2}$, $P_{u_3}$, $P_{u_4}$, and $P_{u_5}$ form a vertical directional component $Q_1$. On columns $x=2,6$, the vertical segments of $P_{u_1}$ and $P_{u_0}$ form vertical directional components $Q_3$ and $Q_2$, respectively.}
  \label{fig2}
\end{figure}

\begin{lemma}\label{le2.3}
If $G$ is chordal, then $T_{\mathcal{R}}$ is a forest. If $G$ is connected, then $T_{\mathcal{R}}$ is a tree.
\end{lemma}
\begin{proof}
\par We prove the forest assertion by induction on $|V(G)|$, with the empty graph as the base case. Let $v$ be a simplicial vertex of a nonempty chordal graph, and $v\in H\cap Q$.

\par Deleting $v$ does not split any remaining horizontal component. If $v$ has at most one horizontal neighbor, removing $v$ cannot split a horizontal component. If $v$ has two or more horizontal neighbors, each of them shares a horizontal grid edge with $P_v$, so they all lie in the same component $H$ of $G_{\mathcal{H}}$ as $v$. Since $v$ is simplicial, they are pairwise adjacent in $G$, and Lemma \ref{le2.2} gives $G_{\mathcal{H}}[H]=G[H]$; hence they are pairwise adjacent in $G_{\mathcal{H}}$. Any path in $G_{\mathcal{H}}$ that uses $P_v$ internally can therefore bypass $P_v$ through an edge joining its predecessor and successor. The same reasoning applies to vertical components. Thus each nonempty component after deletion is exactly an old component with $v$ removed.

\par The clique $N[v]$ has a common grid edge by Lemma \ref{le2.2}. Suppose this edge is horizontal; the vertical case follows by interchanging the two directions. Every neighbor of $v$ belongs to $H$. In particular, each vertical neighbor of $v$ belongs to $H\cap Q=C_{HQ}$.

\par If $|C_{HQ}|\geq2$, the sets $H$, $Q$, and $C_{HQ}$ remain nonempty after deletion. No directional component splits, and every other component intersection remains unchanged. Consequently, the incidence graph of the restricted representation is exactly $T_{\mathcal{R}}$. The induction hypothesis proves that it is a forest.

\par If $C_{HQ}=\{v\}$, then $v$ has no vertical neighbor. Hence $Q=\{v\}$. The node $Q$ is a leaf of $T_{\mathcal{R}}$. If $H\ne\{v\}$, deleting $v$ removes precisely this leaf and its incident edge. If $H=\{v\}$, then $v$ is isolated in $G$, and deleting it removes a $K_2$ of $T_{\mathcal{R}}$. Reattaching a leaf or restoring an isolated edge cannot create a cycle. This completes the induction.

\par Finally, suppose $G$ is connected. Adjacent vertices belong to incidence edges sharing a horizontal or a vertical component node. A graph path of $G$ therefore gives a sequence of incidence edges in one connected component of $T_{\mathcal{R}}$. Every incidence edge has a nonempty label, and every incidence node has an incident edge. Hence $T_{\mathcal{R}}$ is connected.
\end{proof}

\par The nodes of this incidence graph are directional components, not maximal cliques. Several maximal cliques may lie in one component.

\begin{definition}\label{def2.1}
An {\em ordered clique partition tree} of a nonempty graph $G$ consists of a tree $T$ with at least one edge and nonempty sets $C_e$, indexed by $e\in E(T)$, with the following properties. 
\par (1) The sets $C_e$ partition $V(G)$, and every $C_e$ is a clique.
\par (2) If $e$ and $f$ have no common endpoint in $T$, then $C_e$ and $C_f$ are non-adjacent.
\par (3) For each $t\in V(T)$, its incident edges have an order $e_1,\ldots,e_d$ such that, for every $v\in C_{e_i}$, an integer $p_t(v)\in\{i,\ldots,d\}$ satisfies
\begin{equation}\label{eq1}
  N_G(v)\cap\bigcup_{j=i+1}^{d}C_{e_j}=\bigcup_{j=i+1}^{p_t(v)}C_{e_j}.
\end{equation}
The union on the right is empty when $p_t(v)=i$.
\end{definition}

\par The cliques $C_e$ need not be maximal, and their vertices need not be twins. Different vertices in one $C_e$ may have different values of $p_t(v)$. The requirement in \eqref{eq1} is that $v$ is adjacent to every vertex of an initial segment of the later cliques $C_{e_{i+1}},\cdots,C_{e_{p_t(v)}}$, and to nothing beyond. Consider $H_1$ in $T_{\mathcal{R}}$ with $\mathcal{R}=\langle F_0,P\rangle$, the order of its neighbors is $Q_3,Q_1,Q_2$.

\begin{lemma}\label{le2.4}
Let $T$ be a bipartite tree with parts $\mathcal{H}$ and $\mathcal{V}$. Prescribe a linear order of the neighbors of each node. There is a linear order of $\mathcal{H}$ and a linear order of $\mathcal{V}$ whose restrictions realize all the prescribed neighbor orders.
\end{lemma}
\begin{proof}
Induct on the number of nodes, allowing a one-node tree as the base case. Remove a leaf $z$ with neighbor $t$, and apply induction to the remaining tree, deleting $z$ from the prescribed order at $t$. Suppose $z\in\mathcal{V}$. In the resulting order of $\mathcal{V}\setminus\{z\}$, the neighbors of $t$ occur in their prescribed relative order. Insert $z$ between its predecessor and successor in that neighbor order. If one of them does not exist, insert $z$ before the first neighbor or after the last; if neither exists, any position is suitable. No other node has $z$ as a neighbor, so this insertion violates no other prescription. The order at $z$ contains only $t$ and is automatic. The case $z\in\mathcal{H}$ is identical.
\end{proof}

\begin{theorem}\label{th2.1}
A connected graph $G$ is a chordal $\llcorner$-EPG graph if and only if it has an ordered clique partition tree. Given such a partition for $G$, one can construct an $\llcorner$-EPG representation whose endpoints have integer coordinates in $[0,2|V(G)|+1]^2$.
\end{theorem}
\begin{proof}
\par Suppose first that $G$ is chordal and has an $\llcorner$-EPG representation $\mathcal{R}$. Apply Lemma \ref{le2.1} and take $T=T_{\mathcal{R}}$. Lemma \ref{le2.3} shows that $T$ is a tree. The sets $C_{HQ}=H\cap Q$ are nonempty cliques and partition $V(G)$. If vertices from two different sets are adjacent, they share a horizontal or vertical component, so the corresponding tree edges have a common endpoint. This proves the first two partition conditions.

\par Fix $H\in\mathcal{H}$. Write its adjacent vertical components as $Q_1,\ldots,Q_d$ in increasing order of their column coordinates $x_1,\ldots,x_d$. These coordinates are distinct. Otherwise, paths from $H\cap Q_i$ and $H\cap Q_j$ would have the same bend and would share a vertical edge, forcing $Q_i=Q_j$. Put $e_i=HQ_i$. For $u\in C_{e_i}$, $v\in C_{e_j}$, and $i<j$,
\[
uv\in E(G)\iff r_u>x_j.
\]
The condition does not depend on the choice of $v$ within $C_{e_j}$. Since the $x_j$ increase strictly, the later cliques adjacent to $u$ are precisely an initial segment. This gives \eqref{eq1}. Ordering the horizontal components at a vertical node from bottom to top proves the corresponding assertion there.

\par Conversely, suppose that an ordered clique partition tree is given. We first prove chordality. For $t\in V(T)$, set 
\[
S_t=\bigcup_{e\ni t}C_e,
\]
and use the incident order $e_1,\ldots,e_d$ at $t$. Assign to $v\in C_{e_i}$ the interval 
\[
I_v^t=[2i,2p_t(v)+1].
\]
Intervals belonging to the same $C_{e_i}$ intersect. If $u\in C_{e_i}$, $v\in C_{e_j}$, and $i<j$, they intersect exactly when $j\leq  p_t(u)$, which is equivalent to $uv\in E(G)$ by \eqref{eq1}. Thus these intervals represent $G[S_t]$.

\par Root $T$ and add its sets $S_t$ in any order in which parents precede children. When a non-root node $t$ is added, its intersection with the previously added vertex set is exactly the clique indexed by its parent edge. The second partition condition excludes edges between the two sides outside this intersection. To verify that gluing preserves chordality, suppose an induced cycle of length at least four meets both sides of a clique separator. It must contain at least two separator vertices. With three or more, two are nonconsecutive on the cycle and give a chord. With exactly two, both sides of the cycle between them contain vertices outside the separator, so their joining edge is again a chord. A cycle contained in one side already has a chord. Since every $G[S_t]$ is an interval graph, and hence chordal, the resulting graph $G$ is chordal.

\par We now construct its $\llcorner$-EPG representation. Choose a bipartition $(\mathcal{H},\mathcal{V})$ of $T$. Interpret $\mathcal{H}$ as row nodes and $\mathcal{V}$ as column nodes. By Lemma \ref{le2.4}, give the column nodes distinct ranks $x(Q)\in\{1,\ldots,|\mathcal{V}|\}$ and the row nodes distinct ranks $y(H)\in\{1,\ldots,|\mathcal{H}|\}$ that respect every prescribed neighbor order.

\par Let $v\in C_{HQ}$. Let the neighbors at $H$ be ordered as $Q_1,\ldots,Q_d$, and at $Q$ write them as $H_1,\ldots,H_s$. Put $k=p_H(v)$ and $\ell=p_Q(v)$. Define
\begin{equation}\label{eq2}
 P_v=[(2x(Q),2y(H)),(2x(Q_k)+1,2y(H))]\cup [(2x(Q),2y(H)),(2x(Q),2y(H_\ell)+1)].
\end{equation}
The lower bound on each prefix index ensures that both arms have length at least one. Thus every bend is a lower-left corner.

\par Two vertices of the same $C_{HQ}$ have the same bend and share the first grid edge of both arms. Next let $u\in C_{HQ_i}$ and $v\in C_{HQ_j}$ with $i<j$. They lie on the same row and on distinct columns, so only a horizontal overlap can produce adjacency. The later arm starts at $2x(Q_j)$, and $P_u$, $P_v$ share a horizontal edge if and only if
\[
2x(Q_j)<2x(Q_{p_H(u)})+1\iff j\leq p_H(u)\iff uv\in E(G).
\]
The vertical calculation is the same, with $x,Q,H$ replaced by $y,H,Q$. If two tree edges have no common endpoint, their paths have distinct rows and distinct columns. They may cross at a point, but they cannot share a grid edge. These cases verify every adjacency and non-adjacency.

\par Finally, the nonempty partition sets imply $|E(T)|\leq |V(G)|$. Each part of a bipartition of a tree with at least one edge has at most $|E(T)|$ nodes. Therefore every coordinate in \eqref{eq2} is at most $2|V(G)|+1$ and is a positive integer.
\end{proof}

\par For disconnected graphs, apply the characterization to each nonempty component and place their representations in disjoint sets of rows and columns. Thus restricting the structural statements to connected nonempty graphs loses no membership cases.

\section{Admissible partial orientations of clique trees}

\par The characterization of Section 2 uses an incidence tree whose nodes are directional components. We now develop a second characterization that works directly with maximal cliques, using a partial orientation of a clique tree. A partial orientation of a tree assigns a direction to some edges and leaves the others undirected. An edge cannot receive both directions. Only directed edges contribute to the in-degree and out-degree. Recall that for a clique tree edge $KQ$, we call $S_{KQ}=K\cap Q$ its separator.

\begin{definition}\label{def3.1}
A partial orientation of a clique tree $\mathcal{T}$ is {\em admissible} if, for every node $K$, its out-degree and in-degree satisfy the following conditions, respectively:
\begin{equation}\label{eq3}
 d^-_{\mathcal{T}}(K)\leq 1,\qquad d^+_{\mathcal{T}}(K)\leq 1,
\end{equation}
and, for any distinct neighbors $Q,R$ of $K$,
\begin{equation}\label{eq4}
 S_{KQ}\cap S_{KR}\ne\emptyset\Longrightarrow K\to Q\text{ or }K\to R.
\end{equation}
The same definition applies to a tree with clique labels $B_u$ at its nodes: the separator on $uv$ is then $B_u\cap B_v$.
\end{definition}

\par We now give one clique tree of a chordal $\llcorner$-EPG graph $K_{1,4}$ that admits no admissible partial orientation, and another that does admit such an orientation.
\par Let the star $K_{1,4}$ have center $v$ and leaves $u_i$ ($i\in[4]$). For $i\in[4]$, let $D^i=\{v,u_i\}$ be a maximal clique of $K_{1,4}$. In the clique tree of Fig. \ref{fig3}(b), for every pair $i\neq j$ in $\{2,3,4\}$, the intersection $S_{D^1D^i}\cap S_{D^1D^j}=\{v\}$ forces $D^1\to D^i$, $D^1\to D^j$ by \eqref{eq4}. A pigeonhole argument over the three pairs shows that at least two of $D^1\to D^2$, $D^1\to D^3$, and $D^1\to D^4$ must hold, contradicting $d^-_{\mathcal{T}}(D^1)\leq 1$. Hence this clique tree does not admit an admissible partial orientation. In the clique tree of Fig. \ref{fig3}(c), the equalities $S_{D^2D^1}\cap S_{D^2D^3}=\{v\}$, and $S_{D^3D^2}\cap S_{D^3D^4}=\{v\}$ imply that at least one of $D^2\to D^1$ and $D^2\to D^3$ holds, and at least one of $D^3\to D^2$ and $D^3\to D^4$ holds. We may therefore set $D^3\to D^2\to D^1$, and this clique tree admits an admissible partial orientation.

\begin{figure}
\centering
  \setlength{\unitlength}{1bp}%
   \begin{picture}(389.66, 94.77)(0,0)
  \put(0,0){\includegraphics{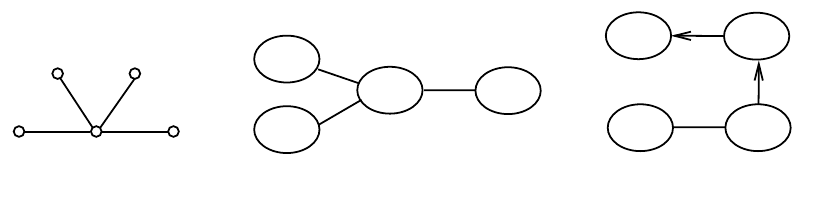}}
  \put(24.56,64.02){\fontsize{9.72}{11.66}\selectfont $u_3$}
  \put(62.58,64.71){\fontsize{9.72}{11.66}\selectfont $u_2$}
  \put(43.27,20.09){\fontsize{9.72}{11.66}\selectfont $v$}
  \put(80.13,20.05){\fontsize{9.72}{11.66}\selectfont $u_1$}
  \put(5.67,20.44){\fontsize{9.72}{11.66}\selectfont $u_4$}
  \put(130.77,64.00){\fontsize{9.72}{11.66}\selectfont $D^2$}
  \put(130.70,29.51){\fontsize{9.72}{11.66}\selectfont $D^4$}
  \put(181.32,49.07){\fontsize{9.72}{11.66}\selectfont $D^1$}
  \put(237.97,48.51){\fontsize{9.72}{11.66}\selectfont $D^3$}
  \put(300.63,75.22){\fontsize{9.72}{11.66}\selectfont $D^1$}
  \put(357.28,74.66){\fontsize{9.72}{11.66}\selectfont $D^2$}
  \put(301.46,30.72){\fontsize{9.72}{11.66}\selectfont $D^4$}
  \put(357.94,30.72){\fontsize{9.72}{11.66}\selectfont $D^3$}
  \put(42.00,7.76){\fontsize{9.72}{11.66}\selectfont (a)}
  \put(182.97,7.76){\fontsize{9.72}{11.66}\selectfont (b)}
  \put(330.35,7.76){\fontsize{9.72}{11.66}\selectfont (c)}
  \end{picture}%
\caption{Let $D^i=\{v,u_i\}$. (b) A clique tree of $K_{1,4}$ does not admit an admissible partial orientation. (c) A clique tree of $K_{1,4}$ admits an admissible partial orientation.}
\label{fig3}
\end{figure}

\par The degree condition alone ensures that the directed edges form internally vertex-disjoint directed paths. The separator condition records which end of a directed path can meet a change of direction. Three preliminary lemmas will allow us to retain this information when replacing local cliques by maximal cliques.

\par The next three lemmas are used, respectively, to unify the ordering of local maximal cliques, to reduce local clique labels to maximal cliques, and to characterize the relationship between the vertex support tree and the orientation.

\begin{lemma}\label{le3.1}
Let an interval graph with a fixed interval representation. Order its maximal cliques by increasing representative points on the line. If a nonempty family of representing intervals has a common left endpoint, all vertices in that family have the same first maximal clique in this order.
\end{lemma}

\begin{proof}
\par The Helly property for intervals gives a common point for each maximal clique. Maximality ensures that the intervals containing that point are exactly the intervals of the clique. Choose one representative point for each maximal clique. Distinct maximal cliques have distinct representative points, and ordering these points makes the cliques containing each vertex consecutive.

\par Let the common left endpoint of the given family be $x_0$, and let $K^i$ be the earliest maximal clique containing any vertex of the family. Its representative point $z_i$ satisfies $z_i\geq x_0$. Suppose an interval $I_w$ in the family does not contain $z_i$. Its right endpoint is then smaller than $z_i$. A maximal clique containing $w$ has a representative point $z_j\in I_w$, so $x_0\leq  z_j<z_i$. This clique precedes $K^i$ and contains a member of the family, contradicting the definition of $K^i$. Therefore every member first occurs in $K^i$.
\end{proof}

\begin{lemma}\label{le3.2}
Let $G$ be connected and nonempty. Suppose a tree $\mathcal{U}$ has nonempty clique labels $B_u$ with the following properties: every graph vertex and edge belongs to a label; the labels containing any fixed vertex induce a connected subtree; every tree-edge separator is nonempty; and $\mathcal{U}$ has an admissible partial orientation. Suppose also that the two labels on each directed edge are incompatible under inclusion, i.e. neither $B_u\subseteq B_v$ nor $B_v\subseteq B_u$, and that every maximal clique of $G$ occurs as a label. Then $\mathcal{U}$ can be reduced to a clique tree of $G$ with an admissible partial orientation.
\end{lemma}

\begin{proof}
Assume that a label $B_u$ is contained in another label $B_w$, with $u\ne w$. Let $a$ be the neighbor of $u$ on the $u$--$w$ path. Running intersection implies $B_u\subseteq B_a$. Since labels on directed edges are incompatible, the edge $ua$ is undirected, and its separator is $B_u$.

\par For any other neighbor $b$ of $u$, since the nonempty separator $S_{bu}=B_u\cap B_b\subseteq B_u=S_{ua}$, we have $S_{bu}\cap S_{ua}\ne\emptyset$ and $u\to b$ by \eqref{eq4}. The out-degree bound permits at most one such neighbor. Thus $u$ is either a leaf or a node of degree two, with its other edge directed away from $u$.

\par If $u$ is a leaf, delete it. If $u$ has degree two with the other neighbor $b$, replace the path $a-u\to b$ by an undirected edge $ab$, and delete $u$. Since $\mathcal{U}$ is a tree, there was no edge $ab$ before this operation. The new graph is again a tree. Running intersection and $B_u\subseteq B_a$ give
\begin{equation}\label{eq5}
 B_a\cap B_b=B_u\cap B_b.
\end{equation}

\par We check each invariant. Since $B_a$ contains $B_u$, deleting $B_u$ loses no vertex or graph edge. Running intersection is preserved: an occurrence of a vertex at $u$ can be replaced by its occurrence at $a$, and \eqref{eq5} gives the required intersection across the new edge. Since $S_{ub}$ was nonempty, the new separator is nonempty.

\par At $a$, the old undirected edge had separator $B_u$, and the new undirected edge has the subset $B_u\cap B_b$. Thus any new pair of intersecting separators already intersected before the operation. At $b$, an incoming edge is replaced by an undirected edge with exactly the same separator. Neither edge is outgoing from $b$, so every covering requirement at $b$ remains unchanged. Directed degrees cannot increase. No new directed edge is introduced, so the incomparability condition on directed edges also persists.

\par Repeat until no label is contained in another label at a distinct node. Each operation removes a node, so the process terminates. Every maximal clique retains at least one occurrence: if its label is deleted, a containing label must equal it. Every remaining nonmaximal clique label would be contained in a retained maximal clique label, which is impossible at termination. Equal repeated labels are impossible for the same reason. Hence the remaining labels are exactly the distinct maximal cliques of $G$, and running intersection makes the remaining tree a clique tree. Its partial orientation is admissible.
\end{proof}

\par Before presenting the next lemma, we introduce a definition. A {\em directed component} is a directed path $K^1\to \cdots \to K^r$.

\begin{lemma}\label{le3.3}
Let $\mathcal{T}$ be a clique tree with an admissible partial orientation. For every graph vertex $v$, the subtree $\mathcal{T}_v$ is a path and contains at most one undirected edge. If it contains an undirected edge $KQ$, its portions in the directed components through $K$ and $Q$ start at these two nodes and follow the directions of those components.
\end{lemma}

\begin{proof}
The subtree $\mathcal{T}_v$ is connected by running intersection. If a node $K$ had three neighbors in $\mathcal{T}_v$, then any two of the corresponding separators would contain $v$. By \eqref{eq4}, every pair of these edges would include an edge directed away from $K$. However, at most one of the three edges can be outgoing. The other two violate \eqref{eq4}. Thus every node of $\mathcal{T}_v$ has degree at most two, so $\mathcal{T}_v$ is a path.

\par Two undirected edges cannot be consecutive on this path, since their separators would intersect at $v$ and neither edge would be outgoing at their common node. Suppose there are two undirected edges, and choose two such that no other undirected edge lies between them on $\mathcal{T}_v$. The nonempty directed portion between them consists of directed edges. At each internal node of this directed portion, \eqref{eq3} excludes both two incoming edges and two outgoing edges, so the whole directed portion has a consistent direction. At the endpoint adjacent to one undirected edge, \eqref{eq4} requires the adjacent directed edge to point into the directed portion. At the endpoint adjacent to the other undirected edge, the same condition requires the opposite direction. This is impossible.

\par If $KQ$ is that edge and $\mathcal{T}_v$ continues at $K$, its next edge must be outgoing from $K$ by \eqref{eq4}. The degree conditions force the subsequent directed edges to continue in the same direction. No incoming edge at $K$ can belong to $\mathcal{T}_v$. The same argument applies at $Q$, proving the final assertion.
\end{proof}

\begin{theorem}\label{th3.1}
A connected nonempty chordal graph $G$ is an $\llcorner$-EPG graph if and only if it has a clique tree with an admissible partial orientation.
\end{theorem}

\begin{proof}
\par Suppose first that $G$ is an $\llcorner$-EPG graph. Its directional component-incidence graph $T_{\mathcal{R}}$ is a tree by Lemma \ref{le2.3}. Each directional component induces an interval graph by Lemma \ref{le2.2}. In every horizontal component, order its local maximal cliques from left to right and join consecutive ones into a directed path. In every vertical component do the same from bottom to top. These paths may consist of one node. We keep distinct nodes when the same clique label occurs in two components.

\par Consider an incidence edge $HQ$ and its clique $C_{HQ}=H\cap Q$. On the row of $H$, all vertices in $C_{HQ}$ have the same left endpoint. On the column of $Q$, all have the same bottom endpoint. By Lemma \ref{le3.1}, these vertices have a common first local maximal clique in each component. Join these two first-clique nodes by an undirected edge. Their intersection is exactly $C_{HQ}$: it contains this set by construction and is contained in $H\cap Q$.

\par Call the resulting clique-labeled graph $\mathcal{U}$. Contracting its directed paths gives $T_{\mathcal{R}}$. Each directed path replaces one node of $T_{\mathcal{R}}$ and is itself a tree; the undirected edges of $\mathcal{U}$ correspond bijectively to the edges of $T_{\mathcal{R}}$. Replacing nodes of a tree by trees, with incident edges attached to the appropriate replacement nodes, preserves the tree property. Hence $\mathcal{U}$ is a tree. Every edge of $E(G)$ belongs to a local maximal clique in a horizontal or vertical component, and every vertex of $V(G)$ belongs to such a clique. For a fixed vertex $v\in C_{HQ}$, its local cliques are consecutive on the path for $H$ and consecutive on the path for $Q$. The undirected edge between their first cliques joins these two portions. No other component contains $v$, so running intersection holds.

\par Every separator of $\mathcal{U}$ is nonempty. This is immediate for undirected edges. For consecutive local maximal cliques it follows from connectedness of the local interval graph: if their intersection were empty, no representing interval could occur on both sides of that position, and the component would be disconnected. Since the labels on each directed edge are distinct maximal cliques of the same induced interval graph, labels on each directed edge are incompatible.

\par We next check admissibility. Each node has at most one incoming and one outgoing directed edge on its local clique path. An outgoing edge covers every pair of incident separators that includes it. The separator of an incoming edge is disjoint from each incident undirected-edge separator: a vertex in the latter has its first local clique at this node and cannot occur in the preceding clique. Since the sets $C_e$ partition $V(G)$, two distinct undirected-edge separators are disjoint. These observations cover all pairs of incident edges and prove \eqref{eq4}.

\par Finally, every maximal clique of $G$ occurs as a label of $\mathcal{U}$. Lemma \ref{le2.2} gives it a common horizontal or vertical grid edge, so it lies within one directional component. It is also maximal in that induced interval graph; otherwise a larger local clique would contradict global maximality. Lemma \ref{le3.2} now converts $\mathcal{U}$ into a clique tree of $G$ with an admissible partial orientation.

\par Conversely, suppose that $\mathcal{T}$ has an admissible partial orientation. Consider the subgraph of $\mathcal{T}$ consisting only of its directed edges. By the degree condition \eqref{eq3}, every directed component of this subgraph is a directed path $K^1 \to\cdots\to K^r$ (possibly one node). Now contract each such directed path to a node. All undirected edges of $\mathcal{T}$ that were incident to nodes of the path become incident to the new node, and thus become the edges of the resulting tree $T_0$.

\par For each undirected edge $e=KQ$ of $\mathcal{T}$, assign the clique $C_e=K\cap Q$ to the corresponding edge of $T_0$. These cliques are nonempty and pairwise disjoint. Indeed, if a vertex belonged to two of them, its support $\mathcal{T}_v$ would contain two undirected edges, contrary to Lemma \ref{le3.3}. That lemma also shows that, on each of the two directed components incident with $e$, all vertices of $C_e$ have their first occurrence at the endpoint of $e$.

\par Every vertex not yet assigned to a clique $C_e$ has its support entirely within a directed component by Lemma \ref{le3.3}. On such a component, set the maximal cliques as $K^1\to\cdots\to K^r$. Group these unassigned vertices according to their first occurrence: for each $i$, let $D_i$ consist of those first appearing in $K^i$. Since $D_i\subseteq K^i$, each nonempty $D_i$ is a clique. For every nonempty group, attach a new private leaf to its component node in $T_0$, and assign $D_i$ to this new edge. Denote the resulting tree by $T$. The assigned sets now partition $V(G)$ into nonempty cliques. If $T_0$ has no edge, all vertices are unassigned before this step, so at least one private leaf is added. Thus $T$ has at least one edge.

\par If two graph vertices are adjacent, some maximal clique contains both. Its directed component is an endpoint of each of their assigned tree edges. Hence sets indexed by non-incident tree edges are non-adjacent.

\par It remains to verify that \eqref{eq1} holds. At a component node corresponding to $K^1\to\cdots\to K^r$, every incident edge-clique has a common first index. This follows from Lemma \ref{le3.3} for the original edges of $T_0$ and from the definition of $D_i$ for the new leaf edges. Order the incident edges $e_1,\ldots,e_d$ by non-decreasing first index, breaking ties arbitrarily. Write $a_i$ for the first index of $C_{e_i}$. For $v\in C_{e_i}$, let $b(v)$ be its last occurrence on this directed component. Its local support is exactly $K^{a_i},\ldots,K^{b(v)}$.

\par Take $j>i$ and $w\in C_{e_j}$. Since $a_j\geq a_i$, these two support intervals meet exactly when $a_j\leq  b(v)$. Moreover, vertices from the two distinct edge-cliques cannot meet in any other directed component: distinct edges of the tree $T$ have at most one common endpoint, and a private leaf corresponds to no additional component. Therefore
\[
vw\in E(G)\iff a_j\leq  b(v).
\]
This condition is independent of the choice of $w$ and holds for a consecutive prefix of $j=i+1,\ldots,d$. Since $a_i\leq  b(v)$, the largest index $j$ with $a_j\leq  b(v)$ is at least $i$; take it as $p_t(v)$. At a private leaf, the only possible value is $p_t(v)=1$. Thus all parts of Definition \ref{def2.1} hold. Theorem \ref{th2.1} supplies the required representation.
\end{proof}

\par The two constructions have different roles. The incidence tree of Section 2 records directional components of a representation; the tree in Theorem \ref{th3.1} has one node for each maximal clique of the graph. Lemma \ref{le3.2} is what connects these objects. In the reverse construction, vertices within one edge-clique may have different last occurrences, which is why the partition condition does not require them to be twins.

\section{An example: from representation to clique tree}

\par We trace both characterizations through an example, starting from an $\llcorner$-EPG representation and ending with the clique tree $\mathcal{T}$ of Theorem \ref{th3.1} and the incidence tree $T_{\mathcal{R}}$ of Theorem \ref{th2.1}.

\par The graph $F_0$ (see Fig. \ref{fig1} left) is obtained by joining a vertex $v$ to all vertices of $P_6=u_0 u_1 \cdots u_5$. Its maximal cliques are $K^i = \{v, u_i, u_{i+1}\}$ for $0 \leq i \leq 4$.


\par \textbf{The $\llcorner$-EPG representation.}  Let 
\begin{align*}
P_v     &= [(4,2),(7,2)] \cup [(4,2),(4,9)];\quad P_{u_0}= [(6,2),(7,2)] \cup [(6,2),(6,3)];\\
P_{u_1} &= [(2,2),(7,2)] \cup [(2,2),(2,3)];\quad P_{u_2}= [(4,2),(5,2)] \cup [(4,2),(4,5)];\\
P_{u_3} &= [(4,4),(5,4)] \cup [(4,4),(4,7)];\quad P_{u_4}= [(4,6),(5,6)] \cup [(4,6),(4,9)];\\
P_{u_5} &= [(4,8),(5,8)] \cup [(4,8),(4,9)].
\end{align*}
This gives an $\llcorner$-EPG representation of $F_0$ (see Fig. \ref{fig1}, right).

\par \textbf{Directional components.} By Lemma \ref{le2.2}, each directional component induces an interval graph. The horizontal components are $H_1 =\{v,u_0,u_1,u_2\}$ (row 2), $H_2 = \{u_3\}$, $H_3 = \{u_4\}$, and $H_4 = \{u_5\}$, and the vertical components are $Q_1 = \{v, u_2, u_3, u_4, u_5\}$ (column 4), $Q_2 = \{u_0\}$, and $Q_3 = \{u_1\}$.

\par \textbf{First route: the incidence tree $T_{\mathcal{R}}$ (Theorem \ref{th2.1}).}
The bipartite incidence graph $T_{\mathcal{R}}$ on $\mathcal{H} \dot\cup \mathcal{V}$ joins $H$ to $Q$ when $H \cap Q \ne \emptyset$. Its edges and labels $C_{HQ} = H \cap Q$ are:
\[
H_1 Q_1=\{v, u_2\},H_1 Q_2=\{u_0\},H_1 Q_3=\{u_1\},Q_1 H_2=\{u_3\},Q_1 H_3=\{u_4\},Q_1 H_4=\{u_5\}
\]
The six labels are nonempty, pairwise disjoint, and partition $V(G)$. By Lemma \ref{le2.3}, $T_{\mathcal{R}}$ is a tree (Fig. \ref{fig2}). Taking $T = T_{\mathcal{R}}$ and $C_e = C_{HQ}$ gives an ordered clique partition tree (Definition \ref{def2.1}): at $H_1$, order the neighbors $Q_3, Q_1, Q_2$ by increasing column coordinate, and the condition (3) follows from the monotonicity of left endpoints $r_v$. 

\begin{figure}
\centering
  \setlength{\unitlength}{1bp}%
   \begin{picture}(436.24, 74.65)(0,0)
  \put(0,0){\includegraphics{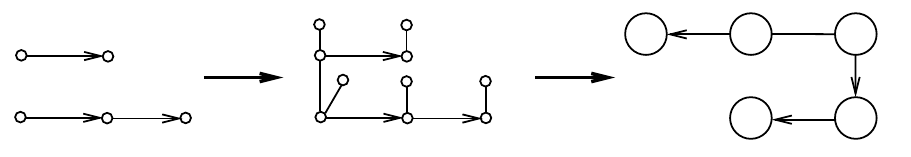}}
  \put(6.09,37.40){\fontsize{9.72}{11.66}\selectfont $A$}
  \put(45.29,37.40){\fontsize{9.72}{11.66}\selectfont $B$}
  \put(5.67,7.76){\fontsize{9.72}{11.66}\selectfont $C$}
  \put(44.87,7.76){\fontsize{9.72}{11.66}\selectfont $D$}
  \put(84.12,7.76){\fontsize{9.72}{11.66}\selectfont $E$}
  \put(198.58,45.81){\fontsize{9.72}{11.66}\selectfont $B$}
  \put(149.79,7.76){\fontsize{9.72}{11.66}\selectfont $C$}
  \put(188.99,7.76){\fontsize{9.72}{11.66}\selectfont $D$}
  \put(228.25,7.76){\fontsize{9.72}{11.66}\selectfont $E$}
  \put(139.66,46.28){\fontsize{9.72}{11.66}\selectfont $A$}
  \put(167.95,33.40){\fontsize{9.72}{11.66}\selectfont $F$}
  \put(199.32,32.86){\fontsize{9.72}{11.66}\selectfont $G$}
  \put(237.33,33.58){\fontsize{9.72}{11.66}\selectfont $I$}
  \put(199.47,61.31){\fontsize{9.72}{11.66}\selectfont $J$}
  \put(156.65,61.39){\fontsize{9.72}{11.66}\selectfont $L$}
  \put(303.80,55.23){\fontsize{9.72}{11.66}\selectfont $K^0$}
  \put(354.16,55.23){\fontsize{9.72}{11.66}\selectfont $K^1$}
  \put(404.52,55.23){\fontsize{9.72}{11.66}\selectfont $K^2$}
  \put(354.16,14.94){\fontsize{9.72}{11.66}\selectfont $K^3$}
  \put(404.52,14.94){\fontsize{9.72}{11.66}\selectfont $K^4$}
  \end{picture}%
  \caption{Steps of the second route in Section 4 within the $\llcorner$-EPG representation $\mathcal{R}=\langle F_0,P\rangle$.}
  \label{fig4}
\end{figure}

\par \textbf{Second route: the labeled tree $\mathcal{U}$ and the clique tree $\mathcal{T}$ (Theorem \ref{th3.1}).} Within each non-singleton component, order the local maximal cliques by increasing representative point and join consecutive ones into a directed path.

\par In $H_1$ (row~2), the local maximal cliques, ordered by increasing left point, are
\[
A = \{v, u_1, u_2\}, \qquad B = \{v, u_0, u_1\},
\]
giving the directed path $A \to B$  (see Fig. \ref{fig4} left).

\par In $Q_1$ (column~4), the local maximal cliques, ordered by increasing bottom point, are
\[
C = \{v, u_2, u_3\}, \qquad D = \{v, u_3, u_4\}, \qquad E = \{v, u_4, u_5\},
\]
giving the directed path $C \to D \to E$  (see Fig. \ref{fig4} left).

\par Each singleton component contributes a node:
\[
F = \{u_3\},\ G = \{u_4\},\ I = \{u_5\},\ J = \{u_0\},\ L = \{u_1\}.
\]

\par By Lemma \ref{le3.1}, the vertices in $C_{HQ} = H \cap Q$ have a common first local maximal clique in each component. Joining the two first-clique nodes by an undirected edge gives:
\[
A\text{--}C \text{ (from } C_{H_1 Q_1} = \{v, u_2\}\text{)},\quad
B\text{--}J \text{ (from } C_{H_1 Q_2} = \{u_0\}\text{)},\quad
A\text{--}L \text{ (from } C_{H_1 Q_3} = \{u_1\}\text{)},
\]
\[
C\text{--}F \text{ (from } C_{H_2 Q_1} = \{u_3\}\text{)},\quad
D\text{--}G \text{ (from } C_{H_3 Q_1} = \{u_4\}\text{)},\quad
E\text{--}I \text{ (from } C_{H_4 Q_1} = \{u_5\}\text{)}.
\]

\par The resulting tree $\mathcal{U}$ has ten nodes, three directed edges, and six undirected edges (see Fig. \ref{fig4} middle). Each directed edge joins two distinct maximal cliques of the same induced interval graph, so the two labels are incompatible under inclusion. Admissibility is immediate: every node has in-degree at most one and out-degree at most one, and every pair of incident separators whose intersection is nonempty includes an outgoing edge.

\par Each singleton node is a leaf whose label is contained in its neighbor's label:
\[
J = \{u_0\} \subseteq B, \quad L = \{u_1\} \subseteq A, \quad F = \{u_3\} \subseteq C, \quad G = \{u_4\} \subseteq D, \quad I = \{u_5\} \subseteq E.
\]
By Lemma \ref{le3.2}, deleting these five leaves yields a clique tree of $G$ with an admissible partial orientation (see Fig. \ref{fig4} right). The remaining nodes are
\[
A = K^1,\quad B = K^0,\quad C = K^2,\quad D = K^3,\quad E = K^4,
\]
and the partial orientation is
\[
K^0 \gets K^1\text{ -- }K^2 \to K^3 \to K^4.
\]

\begin{figure}
\centering
  \setlength{\unitlength}{1bp}%
 \begin{picture}(489.18, 97.36)(0,0)
  \put(0,0){\includegraphics{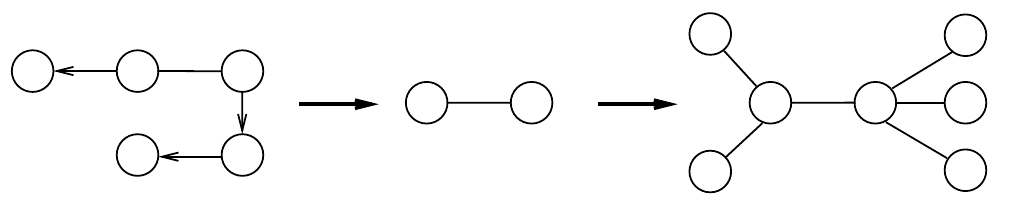}}
  \put(198.53,45.27){\fontsize{9.72}{11.66}\selectfont $H_1$}
  \put(248.89,45.27){\fontsize{9.72}{11.66}\selectfont $Q_1$}
  \put(363.54,45.27){\fontsize{9.72}{11.66}\selectfont $H_1$}
  \put(413.90,45.27){\fontsize{9.72}{11.66}\selectfont $Q_1$}
  \put(334.66,78.28){\fontsize{9.72}{11.66}\selectfont $Q_2$}
  \put(334.66,12.25){\fontsize{9.72}{11.66}\selectfont $Q_3$}
  \put(457.16,45.27){\fontsize{9.72}{11.66}\selectfont $H_3$}
  \put(457.16,77.67){\fontsize{9.72}{11.66}\selectfont $H_2$}
  \put(457.16,12.86){\fontsize{9.72}{11.66}\selectfont $H_4$}
  \put(9.39,60.45){\fontsize{9.72}{11.66}\selectfont $K^0$}
  \put(59.75,60.45){\fontsize{9.72}{11.66}\selectfont $K^1$}
  \put(110.10,60.45){\fontsize{9.72}{11.66}\selectfont $K^2$}
  \put(59.75,20.17){\fontsize{9.72}{11.66}\selectfont $K^3$}
  \put(110.10,20.17){\fontsize{9.72}{11.66}\selectfont $K^4$}
  \end{picture}%
  \caption{The relationship between the incidence tree $T_{\mathcal{R}}$ and the clique tree $\mathcal{T}$: contracting the directed paths of $\mathcal{T}$ and adding the private leaves yields $T_{\mathcal{R}}$.}
  \label{fig5}
\end{figure}

\par \textbf{The relation between $\mathcal{T}$ and $T_{\mathcal{R}}$.} Contracting each directed path of $\mathcal{U}$ into a node identifies $A \to B$ with $H_1$ and $C \to D \to E$ with $Q_1$  (see Fig. \ref{fig5} middle). The five singleton nodes are already in bijection with $H_2, H_3, H_4, Q_2, Q_3$. Under this contraction, the undirected edges of $\mathcal{U}$ become the edges of $T_{\mathcal{R}}$:
\[
A\text{--}C \mapsto H_1\text{--}Q_1, \quad B\text{--}J \mapsto H_1\text{--}Q_2, \quad \ldots, \quad E\text{--}I \mapsto Q_1\text{--}H_4.
\]
Thus, contracting the directed paths of $\mathcal{T}$ and adding the private leaves yields $T_{\mathcal{R}}$ (see Fig. \ref{fig5}, right). We emphasize that this observation is not used in either proof: $T_{\mathcal{R}}$ is constructed directly from the geometry by Lemma \ref{le2.3}, while $\mathcal{T}$ is obtained from $\mathcal{U}$ by Lemma \ref{le3.2}. We record it only to illustrate how the two characterizations relate to each other.

\section{A 2-SAT test for a prescribed clique tree}

\par Theorem \ref{th3.1} characterizes chordal $\llcorner$-EPG graphs by the existence of an admissible partial orientation on some clique tree. We now show that, for a fixed clique tree, testing admissibility reduces to 2-SAT.

\par Fix a clique tree $\mathcal{T}$ of a connected chordal graph $G$. For each edge $KQ$, introduce two distinct Boolean variables $X_{KQ}$ and $X_{QK}$. The value $X_{KQ}=1$ means that the edge is directed from $K$ to $Q$. The assignments $(X_{KQ},X_{QK})=(1,0)$ and $(0,1)$ specify the two directions, while $(0,0)$ leaves the edge undirected. In particular, $X_{QK}$ is not the negation of $X_{KQ}$.

\par Construct a formula $\psi(\mathcal{T})$ by taking the conjunction of four types of clauses. First, for every edge $KQ$, forbid its simultaneous orientation in both directions:
\begin{equation}\label{eq6}
 \neg X_{KQ}\vee\neg X_{QK}.
\end{equation}
Second, for each node $K$ and each unordered pair of distinct neighbors $Q,R$, impose the out-degree restriction
\begin{equation}\label{eq7}
 \neg X_{KQ}\vee\neg X_{KR}.
\end{equation}
Third, for the same pairs, impose the in-degree restriction
\begin{equation}\label{eq8}
 \neg X_{QK}\vee\neg X_{RK}.
\end{equation}
Finally, whenever $S_{KQ}\cap S_{KR}\ne\emptyset$, impose the separator-covering condition
\begin{equation}\label{eq9}
 X_{KQ}\vee X_{KR}.
\end{equation}
Every clause has two literals, so this is a 2-SAT instance. There is no clause $X_{KQ}\vee X_{QK}$: such a clause would incorrectly forbid an undirected edge.

\begin{theorem}\label{th4.1}
The formula $\psi(\mathcal{T})$ is satisfiable if and only if the prescribed clique tree $\mathcal{T}$ has an admissible partial orientation. If $\mathcal{T}$ has $M$ nodes, set
\[
\ell_1=\sum_{K\in V(\mathcal{T})}\binom{d_{\mathcal{T}}(K)}2,\quad \ell_2=\sum_{K\in V(\mathcal{T})}\bigl|\{\{Q,R\}\subseteq N_{\mathcal{T}}(K)\mid S_{KQ}\cap S_{KR}\ne\emptyset\}\bigr|.
\]
Then the direct formula has $2(M-1)$ variables and $(M-1)+2\ell_1+\ell_2$ clauses, with $\ell_2\leq \ell_1\leq \binom{M-1}{2}$. Once the formula is constructed, one can test satisfiability and recover an admissible partial orientation in $O(M^2)$ time. With explicit vertex sets for the cliques, the compatibility test takes $O(nM^2)$ time, where $n=|V(G)|$.
\end{theorem}

\begin{proof}
\par Given an admissible partial orientation, assign 1 to the variables corresponding to its directed edges and 0 to all others. A tree edge has at most one direction, so \eqref{eq6} holds. The directed-degree bounds give \eqref{eq7} and \eqref{eq8}. Condition \eqref{eq4} gives \eqref{eq9}. Thus the assignment satisfies the entire formula.

\par Conversely, let an assignment satisfy $\psi(\mathcal{T})$. Direct $KQ$ from $K$ to $Q$ when $X_{KQ}=1$, direct it the other way when $X_{QK}=1$, and leave it undirected when both variables are 0. Clause \eqref{eq6} makes this rule well-defined. If a node had two outgoing edges, their clause \eqref{eq7} would fail; if it had two incoming edges, their clause \eqref{eq8} would fail. Clause \eqref{eq9} gives every required separator-covering condition. The partial orientation is therefore admissible.

\par There are two variables for each of the $M-1$ tree edges. Clauses of type \eqref{eq6} contribute $M-1$, the out-degree and in-degree clauses contribute $\ell_1$ each, and the covering clauses contribute $\ell_2$ clauses. The number $\ell_1$ counts unordered pairs of tree edges with a common endpoint. Any pair of distinct tree edges has at most one common endpoint, so $\ell_1\leq \binom{M-1}{2}$. The definition gives $\ell_2\leq\ell_1$. Thus the formula has quadratic size.

\par A 2-SAT formula can be solved in time linear in its number of variables and clauses by the implication-graph algorithm \cite{aspvall1979linear}. This yields the $O(M^2)$ bound. Represent each clique by its membership vector on $V(G)$. Every separator and every separator-intersection test then takes $O(n)$ time, and there are $O(M^2)$ tests. This proves the stated compatibility bound. The one-node clique tree has an empty formula and is feasible without any orientation choice.
\end{proof}

\section{Minimal obstructions and graphs with a unique clique tree}
\par It follows from Observation 21 of \cite{cameron2016edge} that every chordal $\llcorner$-EPG graph is strongly chordal. In this section, we present a graph $H$ that is strongly chordal but not $\llcorner$-EPG; its clique tree admits no admissible partial orientation. Moreover, $H$ is a minimal forbidden induced subgraph for $\llcorner$-EPG graphs. Finally, we seek a class of graphs for which we need not identify a clique tree admitting an admissible partial orientation, namely a class of graphs with a unique clique tree.

\begin{figure}
\centering
  \setlength{\unitlength}{1bp}%
  \begin{picture}(429.46, 123.49)(0,0)
  \put(0,0){\includegraphics{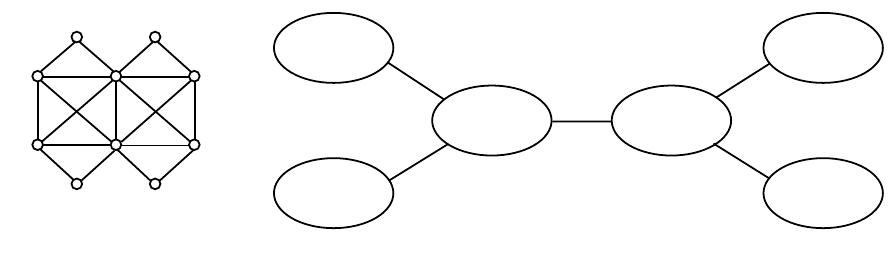}}
  \put(48.60,7.76){\fontsize{9.72}{11.66}\selectfont $H$}
  \put(67.21,110.11){\fontsize{9.72}{11.66}\selectfont $x$}
  \put(52.44,93.45){\fontsize{9.72}{11.66}\selectfont $a$}
  \put(52.32,42.54){\fontsize{9.72}{11.66}\selectfont $b$}
  \put(96.24,84.82){\fontsize{9.72}{11.66}\selectfont $p$}
  \put(96.45,51.42){\fontsize{9.72}{11.66}\selectfont $q$}
  \put(6.54,85.19){\fontsize{9.72}{11.66}\selectfont $r$}
  \put(5.67,51.42){\fontsize{9.72}{11.66}\selectfont $s$}
  \put(30.36,110.11){\fontsize{9.72}{11.66}\selectfont $z$}
  \put(66.78,26.73){\fontsize{9.72}{11.66}\selectfont $y$}
  \put(29.93,26.73){\fontsize{9.72}{11.66}\selectfont $w$}
  \put(222.44,62.50){\fontsize{9.72}{11.66}\selectfont $a,b,r,s$}
  \put(305.21,62.50){\fontsize{9.72}{11.66}\selectfont $a,b,p,q$}
  \put(383.63,97.36){\fontsize{9.72}{11.66}\selectfont $a,p,x$}
  \put(383.63,27.63){\fontsize{9.72}{11.66}\selectfont $b,q,y$}
  \put(148.65,97.36){\fontsize{9.72}{11.66}\selectfont $a,r,z$}
  \put(148.65,27.63){\fontsize{9.72}{11.66}\selectfont $b,s,w$}
  \put(317.61,85.97){\fontsize{9.72}{11.66}\selectfont $A$}
  \put(232.90,85.97){\fontsize{9.72}{11.66}\selectfont $B$}
  \put(353.27,98.23){\fontsize{9.72}{11.66}\selectfont $X$}
  \put(353.27,23.92){\fontsize{9.72}{11.66}\selectfont $Y$}
  \put(189.44,98.23){\fontsize{9.72}{11.66}\selectfont $Z$}
  \put(189.81,26.15){\fontsize{9.72}{11.66}\selectfont $W$}
  \end{picture}%
  \caption{The graph $H$ and its unique clique tree.}
  \label{fig6}
\end{figure}

\par Recall that a vertex $v$ is {\em simple} if the closed neighborhoods of the vertices in $N[v]$ are linearly ordered by inclusion, i.e., 
\[
\forall x,y\in N[v],\quad N[x]\subseteq N[y]\text{ or }N[y]\subseteq N[x].
\]
\begin{theorem}[\cite{farber1983characterizations}]\label{th6.1}
A graph is strongly chordal if and only if it has an ordering in which each vertex is simple in the graph induced by that vertex and its successors.
\end{theorem}

\par Let $H$ have vertex set $\{a,b,p,q,r,s,x,y,z,w\}$. Begin with two complete graphs on $A=\{a,b,p,q\}$ and $B=\{a,b,r,s\}$, with no edges between $\{p,q\}$ and $\{r,s\}$. Add four independent vertices $x,y,z,w$ with neighborhoods $N_H(x)=\{a,p\}$, $N_H(y)=\{b,q\}$, $N_H(z)=\{a,r\}$, and $N_H(w)=\{b,s\}$. $H$ is shown in Fig. \ref{fig6}. The other four maximal cliques are $X=\{a,p,x\}$, $Y=\{b,q,y\}$, $Z=\{a,r,z\}$, and $W=\{b,s,w\}$. Each added vertex is simplicial. Deleting these four vertices leaves two complete graphs joined along the clique $\{a,b\}$, so $H$ is chordal.

\begin{lemma}\label{le6.1}
The unique clique tree of $H$ has edge set $\{AB,AX,AY,BZ,BW\}$. It admits no admissible partial orientation.
\end{lemma}
\begin{proof}
\par The vertex $p$ belongs to exactly the maximal cliques $A$ and $X$. The running-intersection property therefore forces $AX$ to be an edge of every clique tree. The vertices $q,r,s$ similarly force $AY,BZ,BW$, respectively. Every node on the clique tree path between $A$ and $B$ must contain both $a$ and $b$. Only $A$ and $B$ contain both vertices, so $AB$ is also forced. These five edges form a tree on the six maximal cliques. It satisfies the running-intersection property: the cliques containing $a$ form the path $X-A-B-Z$, those containing $b$ form $Y-A-B-W$, each of $p,q,r,s$ occurs in its forced adjacent pair, and each of $x,y,z,w$ occurs in one clique. This proves both existence and uniqueness. Its clique tree is shown in Fig. \ref{fig6} right.

\par At $A$, the three separators are $S_{AB}=\{a,b\}$, $S_{AX}=\{a,p\}$, and $S_{AY}=\{b,q\}$. The nonempty intersections $S_{AB}\cap S_{AX}\neq\emptyset$ and $S_{AB}\cap S_{AY}\neq\emptyset$ force $A\to B$. By the same reasoning at $B$, the nonempty intersections among its three separators force $B\to A$. These two directions on $AB$ are incompatible.
\end{proof}

\begin{theorem}\label{th6.2}
$H$ is a strongly chordal graph and is a minimal forbidden induced subgraph for $\llcorner$-EPG graphs.
\end{theorem}
\begin{proof}
Theorem \ref{th3.1} and Lemma \ref{le6.1} show that $H$ is not a $\llcorner$-EPG graph. We first prove that $H-v$ is an $\llcorner$-EPG graph for every vertex $v$, and then verify strongly chordality. By symmetry, it suffices to treat $H-a$, $H-p$, and $H-x$. For each graph, we give all maximal cliques, verify a clique tree, and check its partial orientation.

 \par \textbf{Deleting $a$.} The maximal cliques are $A'=\{b,p,q\}$, $B'=\{b,r,s\}$, $X'=\{p,x\}$, $Y'=\{b,q,y\}$, $Z'=\{r,z\}$, and $W'=\{b,s,w\}$. Take the tree with edges $A'B'$, $A'X'$, $A'Y'$, $B'Z'$, and $B'W'$. The cliques containing $b$ form the path $Y'-A'-B'-W'$. The vertices $p,q,r,s$ occur in the adjacent pairs $A'X'$, $A'Y'$, $B'Z'$, and $B'W'$, respectively; each remaining vertex occurs in one clique. Hence this is a clique tree.

\par Orient $B'\to A'\to Y'$ and leave $A'X'$, $B'Z'$, and $B'W'$ undirected. The separators at $A'$ are $\{b\}$, $\{p\}$, and $\{b,q\}$. Only the first and third intersect, and their pair has the outgoing edge $A'\to Y'$. The separators at $B'$ are $\{b\}$, $\{r\}$, and $\{b,s\}$. Only the first and third intersect, and their pair has the outgoing edge $B'\to A'$. Leaves have no pairs to check. The directed edges form one directed path, so both degree bounds hold. The orientation is admissible.

\par \textbf{Deleting $p$.} The maximal cliques are $A'=\{a,b,q\}$, $B=\{a,b,r,s\}$, $X'=\{a,x\}$, $Y=\{b,q,y\}$, $Z=\{a,r,z\}$, and $W=\{b,s,w\}$. Use the tree with edges $ZB$, $BA'$, $A'Y$, $ZX'$, and $BW$. The cliques containing $a$ form $X'-Z-B-A'$, and those containing $b$ form $W-B-A'-Y$. The vertices $q,r,s$ occur in the adjacent pairs $A'Y$, $BZ$, and $BW$, respectively; each of $x,y,z,w$ occurs once. Hence this is a clique tree.

\par Orient $Z\to B\to A'\to Y$ and leave $ZX'$ and $BW$ undirected. At $B$, the separators on $BZ$, $BA'$, and $BW$ are $\{a,r\}$, $\{a,b\}$, and $\{b,s\}$. The middle separator meets both of the others, which are disjoint. The edge $B\to A'$ covers both pairs. At $A'$, the two separators $\{a,b\}$ and $\{b,q\}$ meet, and $A'\to Y$ covers their pair. At $Z$, the separators $\{a,r\}$ and $\{a\}$ meet, and $Z\to B$ covers their pair. All other nodes are leaves. Again the directed edges form one directed path, so the orientation is admissible.

\par \textbf{Deleting $x$.} The maximal cliques are $A,B,Y,Z,W$. Use the tree with edges $AB$, $AY$, $BZ$, and $BW$. The support of $a$ is $A-B-Z$, and the support of $b$ is $Y-A-B-W$. The vertex $p$ occurs only in $A$, the vertices $q,r,s$ occur in the adjacent pairs $AY$, $BZ$, and $BW$, respectively, and the remaining vertices occur in one clique. This is a clique tree.

\par Orient $B\to A\to Y$ and leave $BZ$ and $BW$ undirected. At $A$, the separators $\{a,b\}$ and $\{b,q\}$ intersect, and the edge $A\to Y$ covers their pair. At $B$, the separators on $BA$, $BZ$, and $BW$ are $\{a,b\}$, $\{a,r\}$, and $\{b,s\}$. Only the first intersects either of the others, so $B\to A$ covers all required pairs. The two degree bounds hold, and all other nodes are leaves. The orientation is admissible.

\par Theorem \ref{th3.1} now gives $H-a,H-p,H-x$ are $\llcorner$-EPG graphs. The stated symmetries give $H-v$ is an $\llcorner$-EPG graph for every vertex $v$. Deleting a vertex from a path representation preserves an EPG representation and its permitted orientations, so $\llcorner$-EPG is hereditary. Every proper induced subgraph of $H$ is contained in some $H-v$, and hence belongs to $\llcorner$-EPG graphs. This proves minimality.

\par In $H$, the following inclusions hold: 
\begin{align*}
 N_H[x]\subseteq N_H[p]\subseteq N_H[a],\qquad N_H[y]\subseteq N_H[q]\subseteq N_H[b],\\
 N_H[z]\subseteq N_H[r]\subseteq N_H[a],\qquad N_H[w]\subseteq N_H[s]\subseteq N_H[b].
\end{align*}

\par For example, the first chain is
$\{x,a,p\}\subseteq\{p,a,b,q,x\}\subseteq\{a,b,p,q,r,s,x,z\}$. Since $N_H[x]=\{x,a,p\}$, the first chain makes $x$ simple; the other chains give the same conclusion for $y,z,w$. Restricting closed neighborhoods to an induced subgraph preserves these inclusions whenever the compared vertices remain. Thus $x,y,z,w$ can be deleted in this order as simple vertices. Let $F=H-\{x,y,z,w\}$. The vertices $a$ and $b$ are universal in $F$ and $N_F[a]=N_F[b]=V(F)$. The chain $N_F[p]=N_F[q]\subseteq N_F[a]=N_F[b]$ makes $p$ and $q$ simple. In $F-p$, the closed neighborhood of $q$ is $\{a,b,q\}$, and $a,b$ are still universal, so $q$ is simple. Deleting $q$ leaves the complete graph on $\{a,b,r,s\}$, all of whose vertices are simple in any subsequent induced subgraph. Therefore $x,y,z,w,p,q,r,s,a,b$ is a simple elimination ordering, and $H$ is a strongly chordal graph by Theorem \ref{th6.1}.
\end{proof}

\par For a split graph $G$, we consider a split partition $V(G)=C\dot\cup I$ in which $C$ is a maximal clique, $I$ is nonempty, and every vertex of $I$ has a nonempty neighborhood.

\begin{figure}
\centering
  \setlength{\unitlength}{1bp}%
   \begin{picture}(356.32, 99.06)(0,0)
  \put(0,0){\includegraphics{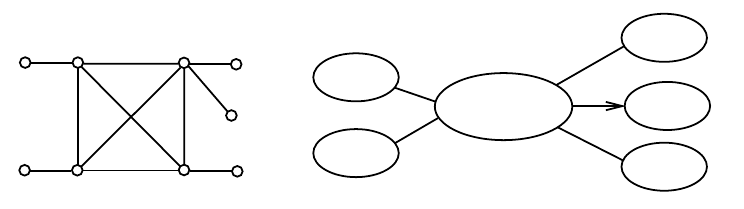}}
  \put(114.95,42.23){\fontsize{9.72}{11.66}\selectfont $s$}
  \put(33.36,76.04){\fontsize{9.72}{11.66}\selectfont $v_1$}
  \put(82.68,76.04){\fontsize{9.72}{11.66}\selectfont $v_2$}
  \put(33.36,7.76){\fontsize{9.72}{11.66}\selectfont $v_4$}
  \put(82.68,7.76){\fontsize{9.72}{11.66}\selectfont $v_3$}
  \put(5.67,75.78){\fontsize{9.72}{11.66}\selectfont $u_1$}
  \put(6.93,7.88){\fontsize{9.72}{11.66}\selectfont $u_4$}
  \put(159.47,60.21){\fontsize{9.72}{11.66}\selectfont $v_1,u_1$}
  \put(159.47,23.83){\fontsize{9.72}{11.66}\selectfont $v_4,u_4$}
  \put(216.34,46.00){\fontsize{9.72}{11.66}\selectfont $v_1,v_2,v_3,v_4$}
  \put(109.23,76.04){\fontsize{9.72}{11.66}\selectfont $u_2$}
  \put(109.23,7.76){\fontsize{9.72}{11.66}\selectfont $u_3$}
  \put(308.61,79.18){\fontsize{9.72}{11.66}\selectfont $v_2,u_2$}
  \put(309.51,17.30){\fontsize{9.72}{11.66}\selectfont $v_3,u_3$}
  \put(308.92,46.46){\fontsize{9.72}{11.66}\selectfont $v_2,s$}
  \end{picture}%
\caption{A split $\llcorner$-EPG graph whose unique clique tree admits an admissible partial orientation.}
\label{fig7}
\end{figure}

\begin{corollary}\label{coro6.1}
Suppose that the neighborhoods of distinct vertices in $I$ are incompatible. Then $G$ is an $\llcorner$-EPG graph if and only if there exists a vertex $s\in I$ such that
\begin{equation}\label{eq10}
 N(u)\cap N(v)=\emptyset
 \quad\text{for all distinct }u,v\in I\setminus\{s\}.
\end{equation}
\end{corollary}

\begin{proof}
The maximal cliques are exactly $C$ and $C_{u}=N[u]$ for $u\in I$. Indeed, a clique contains at most one vertex of $I$. A clique contained in $C$ extends to $C$, and a clique containing $u$ extends to $C_{u}$. Since a larger clique containing $u$ could contain no other independent vertex and no additional vertex of $C$ outside $N(u)$, each $C_{u}$ is maximal. The set $C$ is maximal by hypothesis. The graph is chordal: an ordering of the vertices of $I$, followed by the vertices of $C$, is a perfect elimination order.

\par Fix a clique tree. Every clique on the path from $C_{u}$ to $C$ must contain all of $N(u)$. If $C_{u}$ were not adjacent to $C$, the first node on that path would be some $C_{v}$ with $v\ne u$. Since $N(u)\subseteq C$ and $C_{v}\cap C=N(v)$, the inclusion $N(u)\subseteq C_{v}$ would imply $N(u)\subseteq N(v)$, contrary to incomparability. Hence every $C_{u}$ is adjacent to $C$. The unique clique tree is the star with center $C$ and leaves $C_{u}$.

\par The separator on $CC_{u}$ is $N(u)$. If an admissible orientation exists and some two such separators meet, condition \eqref{eq4} forces an outgoing edge at $C$. Let its endpoint be $C_{s}$. There can be no second outgoing edge at $C$, so every intersecting pair of separators must include $N(s)$. This is \eqref{eq10}. If no two separators meet, any $s\in I$ satisfies that condition.

\par Conversely, suppose \eqref{eq10} holds. Orient $C\to C_{s}$ and leave every other edge undirected. Each pair of intersecting separators at $C$ includes the separator on this outgoing edge. Leaves have no pairs of incident edges, and every in-degree and out-degree is at most one. The partial orientation is admissible, so Theorem \ref{th3.1} gives $G$ is an $\llcorner$-EPG graph. In addition, Fig. \ref{fig7} shows one split $\llcorner$-EPG graph which has a unique clique tree, and this clique tree admits a partial orientation.
\end{proof}

  \begin{figure}
  \centering
  \setlength{\unitlength}{1bp}%
  \begin{picture}(419.18, 125.01)(0,0)
  \put(0,0){\includegraphics{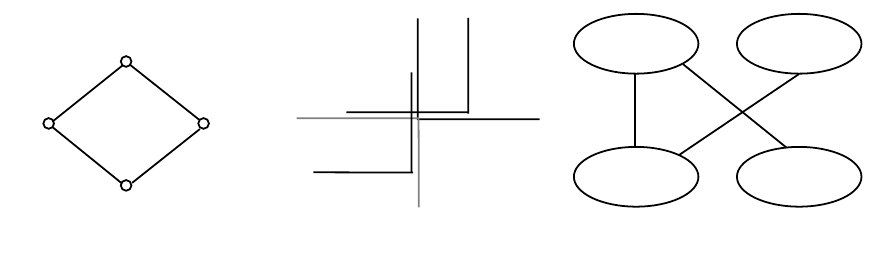}}
  \put(122.72,68.92){\fontsize{9.72}{11.66}\selectfont $P_{v_1}$}
  \put(225.42,100.50){\fontsize{9.72}{11.66}\selectfont $P_{v_2}$}
  \put(148.52,46.74){\fontsize{9.72}{11.66}\selectfont $P_{v_3}$}
  \put(235.06,59.04){\fontsize{9.72}{11.66}\selectfont $P_{v_4}$}
  \put(5.67,65.17){\fontsize{9.72}{11.66}\selectfont $v_1$}
  \put(101.25,63.53){\fontsize{9.72}{11.66}\selectfont $v_4$}
  \put(63.91,94.76){\fontsize{9.72}{11.66}\selectfont $v_3$}
  \put(65.58,31.00){\fontsize{9.72}{11.66}\selectfont $v_2$}
  \put(284.66,103.06){\fontsize{9.72}{11.66}\selectfont $v_1,v_3,v_4$}
  \put(284.66,39.21){\fontsize{9.72}{11.66}\selectfont $v_1,v_2,v_4$}
  \put(377.54,100.50){\fontsize{9.72}{11.66}\selectfont $v_2$}
  \put(377.36,35.80){\fontsize{9.72}{11.66}\selectfont $v_3$}
  \put(48.11,7.76){\fontsize{9.72}{11.66}\selectfont $C_4$}
  \put(174.50,8.86){\fontsize{9.72}{11.66}\selectfont $\mathcal{R}=\langle C_4,P\rangle$}
  \put(332.99,7.76){\fontsize{9.72}{11.66}\selectfont $T_{\mathcal{R}}$}
  \end{picture}%
  \caption{An example showing that relaxing the $\llcorner$-EPG graph condition in Theorem \ref{th2.1} to $B_1$-EPG graphs is not valid: $C_4$ is not a chordal graph, yet the graph $T_{\mathcal{R}}$ is a tree.}
  \label{fig8}
  \end{figure}

\section{Conclusion}
\par In this paper, we have established two equivalent characterizations of chordal $\llcorner$-EPG graphs and a polynomial-time recognition algorithm based on 2-SAT that determines whether a given clique tree admits an admissible partial orientation. We also identify obstructions to the $\llcorner$-EPG property within the class of strongly chordal graphs, and construct a class of graphs with a unique clique tree, eliminating the need to search for a clique tree that admits an admissible partial orientation.
\par The first of these characterizations does not extend from $\llcorner$-EPG graphs to $B_1$-EPG graphs. In Fig. \ref{fig8}, we give $C_4$ together with one of its $B_1$-EPG representations, $\mathcal{R}=\langle C_4,P\rangle$; however, the graph $T_{\mathcal{R}}$ constructed from this representation is a tree. Thus, the fact that $T_{\mathcal{R}}$ is a tree does not necessarily imply that $G$ is chordal.

\par While we fully characterize chordal $\llcorner$-EPG graphs, a polynomial-time recognition algorithm---the second part of the question posed in \cite{cameron2016edge}---remains open.

    \vspace{4mm}
    \noindent\textbf{Acknowledgements} This work was funded in part by the Scientific Research Projects of the Department of Education of Anhui Province (Grants No. 2025AHGXZK40333).
    
    \vspace{1mm}
	\noindent\textbf{Data Availability} Data sharing not applicable to this article as no datasets were generated or analyzed during the current study.
	
	\vspace{1mm}
	
	\noindent\textbf{Conflict of interest} The authors have no relevant financial or non-financial interests to disclose.

\end{document}